\documentclass[10pt,a4paper]{amsart}

\usepackage{amssymb, amsthm, amsfonts, amsmath, mathrsfs, latexsym, pstricks-add, amsrefs, graphicx, tikz, hyperref, thmtools, tkz-graph, enumitem}
\usepackage[mathscr]{euscript}
\usepackage{lmodern, centernot}
\usepackage[T1]{fontenc}

\hypersetup{
	colorlinks,
	linkcolor={red!80!black},
	citecolor={blue!50!black},
	urlcolor={blue!80!black}
	}

\newcommand{\ndiv}{{\hspace{-2pt}\not|\hspace{1pt}}}

\newtheorem{theorem}{Theorem}[section]
\newtheorem{lemma}[theorem]{Lemma}

\newtheorem{proposition}[theorem]{Proposition}

\theoremstyle{definition}

\declaretheorem[name=Example,qed={\lower-0.3ex\hbox{$\triangleleft$}},sibling=theorem]{example}
\newtheorem{remark}[theorem]{Remark}
\newtheorem{question}[theorem]{Question}

\newcommand{\N}{\mathbb N}
\newcommand{\Z}{\mathbb Z}
\newcommand{\Q}{\mathbb Q}
\newcommand{\R}{\mathbb R}

\newcommand{\w}{\omega}
\newcommand{\s}{\sigma}

\newcommand{\Cs}{\mathscr C}

\newcommand{\clorbit}[2]{\overline{\mathcal{O}}_{#2}\left(#1\right)}
\newcommand{\per}[1]{\textrm{\upshape Per}\left(#1\right)}

\newcommand{\notimplies}{\centernot\implies}
\newcommand{\get}[1]{\xleftarrow{\: #1 \:}}
\newcommand{\sg}[2]{#1^{(#2)}}
\newcommand{\inv}[2]{\varprojlim \left(#1,#2\right)}

\begin{document}

\title{Almost Minimal Systems and Periodicity in Hyperspaces}

\author{Leobardo Fern\'andez}
\address{School of Mathematics; University of Birmingham;
Birmingham, B15 2TT, UK}
%    Current address (if needed):
%\curraddr{}
\email{l.fernandez@bham.ac.uk, leobardof@ciencias.unam.mx}
\thanks{The first author acknowledges support by CONACyT,
scholarship for Postdoctoral Position no.~250254}

%    Information for second author:
\author{Chris Good}
\address{School of Mathematics; University of Birmingham;
Birmingham, B15 2TT, UK}
%    Current address (if needed):
%\curraddr{}
\email{c.good@bham.ac.uk}
\thanks{The second and the third author acknowledge support from the European Union through
funding under FP7-ICT-2011-8 project HIERATIC (316705)}

%    Information for third author:
\author{Mate Puljiz}
\address{School of Mathematics; University of Birmingham;
Birmingham, B15 2TT, UK}
%    Current address (if needed):
%\curraddr{}
\email{mxp243@student.bham.ac.uk, puljizmate@gmail.com}


\begin{abstract}
Given a self-map of a compact metric space $X$, we study periodic points of the map induced on the hyperspace of closed subsets of $X$. We give some necessary conditions on admissible sets of periods for these maps. Seemingly unrelated to this, we construct an almost totally minimal homeomorphism of the Cantor set. We also apply our theory to give a full description of admissible period sets for induced maps of the interval maps. The description of admissible periods is also given for maps induced on symmetric products.
\end{abstract}

\maketitle

\section{Introduction}

A continuous function $f\colon X\to X$ from a compact metric space $X$ to itself induces a continuous function $2^f\colon 2^X \to 2^X$ on the (compact metric) space of closed subsets of $X$ with with Hausdorff metric by defining $2^f(C)=f(C)=\{f(x):x\in C\}$ for all closed $C\in 2^X$.  The study of such systems was initiated by Bauer and Sigmund \cite{Bauer}, where they can been seen as a (simplified) topological version of the induced dynamics on the space of measures over $X$ and their evolution under $f$.  Edalat \cite{Edalat} has argued that such systems are the natural approach to dynamical systems, from a computational and domain theoretic point of view.

Clearly one is interested in which properties of the dynamical system $(X,f)$ transfer to the system $(2^X, 2^f)$ and there have been a number of results in this direction.  For example, Banks \cite{Banks} proves that $f$ is (weakly) mixing if and only if $2^f$ is (weakly) mixing and that $2^f$ is transitive if and only of $f$ is weakly mixing.  In \cite{Bauer} and also \cite{TopEntropyForInduced} it is shown that if $f$ has positive entropy, then the entropy of $2^f$ is infinite. Kwietniak and Oprocha go on to show \cite{TopEntropyForInduced} that if $f$ does not have a dense set of recurrent points, then the entropy of $2^f$ is at least $\log 2$. In \cite{ChaosHyp} all the usual notions of chaos except for the classical Devaney's definition is transferred from $f$ to $2^f$, whereas the converse does not hold for any of them. Recently, the authors jointly with Ram{\'{\i}}rez in \cite{ChainTransitivity} have proved that $2^f$ is {chain transitive} if and only if $f$ is {chain weakly mixing} if and only if $f$ is chain transitive and has a fixed point. The first two authors \cite{Shadowing} also show that $2^f$ has shadowing if and only if $f$ does.

However, little work has been done on the periodic points of $2^f$ other than some results on periodicity in symmetric products in  \cite{DynPropOfInduced} which we extend in Section \ref{sec:per}.  It turns out that the theory here is of some interest. Moreover, the techniques that seem to be necessary to construct hyperspaces with various periods are of independent interest. 

Example \ref{ex:cylinder} shows that the set of periods of $2^f$ can be $\N$ even when $f$ has no periodic points at all and in Theorem \ref{tm:interval} we show quite easily that if $f$ is an interval map, then the set of periods of $2^f$ is either $\{1\}$ or $\{1,2\}$ or $\N$.

Trivially, if $x$ is a point if period $n$ under $f$, then $\{x\}$ is a point of period $n$ under $2^f$. Since finite unions of points are again closed sets, it follows, for example, that if $f$ has a period 4 point and a period 6 point, then $2^f$ has points of period 1, 2, 3, 4, 6 and 12. It is possible to give a complete description of admissible periods in $2^f$ that arise in this fashion, i.e.\ those that are formed exclusively of points that were periodic in the original system $(X,f)$, and this is done in Section \ref{sec:per}. More generally if $2^f$ has points of period $4$ and $6$ consisting of pairwise disjoint subsets of $X$, then it would again have points of period 1, 2, 3, 4, 6 and 12. It is reasonable to ask then, whether this is a general property of induced maps on hyperspaces. 

It turns out that this is not the case.  In Section \ref{sec:236} we construct a dynamical system $(Z,g)$ on the Cantor set such that the induced system $(2^Z,2^g)$ only has periodic points of periods $1$, $2$ and $3$. In particular it does not have a period $6$ point. In Theorem \ref{tm:pq} we generalise this construction to an induced system which admits only those periods that are divisors\footnote{Here and elsewhere in text by divisors we actually mean \emph{positive} divisors.} of two positive integers $p$ and $q$.  

To obtain those results we first construct an \emph{almost totally minimal system} over the Cantor set which is of interest in its own right, see Theorem \ref{tm:atm}. Recall that a system $(X,f)$ having a fixed point $x^0\in X$ is almost totally minimal if after removing its (unique) fixed point the remaining non-compact system $(X^*=X\setminus\{x^0\},f)$ is totally minimal, meaning that the full orbit with respect to any iterate of $f$ of any point is dense in $X^*$ (see Section \ref{sec:atm} for a precise definition). This is done via \emph{graph covers}, a tool first devised by Gambaudo and Martens in \cite{GraphCovers} to give a combinatorial description of minimal systems over the Cantor set. This theory has proved useful not just for describing algebraic structure of such systems but also for constructing maps on the Cantor set with particular properties. Shimomura \cite{ShimomuraScrambled}, for example, uses it to construct a transitive, completely scrambled $0$-dimensional system.
There is a close link between these graph covers,  Bratteli-Vershik diagrams and Kakutani-Rokhlin towers and it would be possible to pass from one representation to another. As an example, in the Appendix we give a Bratteli-Vershik representation of that almost totally minimal system.

\medskip

The rest of the paper is organised as follows. In Section \ref{sec:prelim} we briefly introduce the notation. A complete description of the situation for interval maps is given in Section \ref{sec:interval}. In Section \ref{sec:elementary} we introduce minimal building blocks of periodic sets in hyperspaces and derive some fundamental properties related to these. In Section \ref{sec:admis} we seek to characterise admissible periods for induced maps. In Section \ref{sec:atm} an almost totally minimal homeomorphism of the Cantor set is constructed, which we then use in Section \ref{sec:236} to obtain a map whose induced system admits only periods $1$, $2$ and $3$. Section \ref{sec:per} contains results on periods in symmetric products. At the very end in Section \ref{sec:concluding} we list some problems naturally arising from our considerations that are still without a satisfying resolution.

\section{Preliminaries}\label{sec:prelim}

All the spaces under consideration will be compact and metric unless specified otherwise. Given a continuous self-map $f\colon X \to X$ of such a space the inverse limit $\inv{X}{f}$ is defined as the set of full orbits of this system
\begin{equation*}
\inv{X}{f} = \left\{ \bar{x}=(x_i)_{i\in\Z}=(\dots,x_{-1},x_0,x_1,x_2,\dots) \in \prod\limits_{i\in\Z}X \mid f(x_i)=x_{i+1}\right\}.
\end{equation*}
Note the slightly unconventional enumeration of indices above. We shall say that any $\bar{x}\in\inv{X}{f}$ with $x_0=x$ is a \emph{full orbit} of the point $x\in X$. There exists a natural homeomorphism of $\inv{X}{f}$ called (right) shift $\s\colon \inv{X}{f}\to \inv{X}{f}$ given by $(\s (\bar{x}))_i=x_{i+1}$, for all $i \in\Z$. For more on inverse limits see \cite{InvLim}.

Recall that to a full orbit $\bar{x}\in\inv{X}{f}$ once can associate two limit sets, $\w$- and $\alpha$-limit set which are the accumulation sets of the forward and backward orbit of $x_0$ respectively.
\begin{gather*}
\w(\bar{x})=\bigcap_{m=0}^\infty\overline{\bigcup_{n=m}^\infty \{x_{n}\} },\\
\alpha(\bar{x}) = \bigcap_{m=0}^\infty\overline{\bigcup_{n=m}^\infty \{x_{-n}\} }.
\end{gather*}
Both of these are closed and strongly $f$-invariant meaning that $f(\w(\bar{x}))=\w(\bar{x})$ and $f(\alpha(\bar{x}))=\alpha(\bar{x})$, see e.g.\ \cites{BlockCoppel, alpha}. Note that we could equivalently say that they are fixed points of the induced map $2^f$ (see below). As a shorthand we denote their union $\w(\bar{x}) \cup \alpha(\bar{x})$ by $\lambda(\bar{x})$ which is again a closed and strongly $f$-invariant set.

\medskip

The set of all closed (and compact) subsets of $X$ is denoted by $2^X$ and can be metrized using \emph{Hausdorff metric}
\begin{gather*}
H(A,B) = \inf \{ \varepsilon > 0 \mid A \subseteq N_{X}(B, \varepsilon)
\text{ and } B \subseteq N_{X}(A, \varepsilon) \},\\
\text{where\qquad} N_{X}(S, \varepsilon)=\{ x \in X \mid (\exists y\in S)\ d(x,y)<\epsilon \},
\end{gather*}
with the corresponding induced topology coinciding with \emph{Vietoris' topology} (see e.g.\ \cite{HypSpaces} for more details).

The induced map $2^f\colon 2^X\to 2^X$ on the hyperspace $2^X$ is defined by
\begin{equation*}
2^f(S)=f(S)=\{f(s) \mid s\in S \}
\end{equation*}
and is continuous with respect to Hausdorff metric.

\medskip

As we shall constantly be dealing with the sets of periods of different functions it is convenient to introduce a symbol $\per{f}$ for the subset of natural numbers such that $k\in\per{f}$ if and only if there exists a point $x\in X$ with the fundamental period $k$. $\per{2^f}$ will consequentially be the set of all the fundamental periods of points in $2^X$.

We shall also be interested in restrictions of the map $2^f$ to a couple of $2^f$-invariant subsets of $2^X$. We introduce a special symbol for both of these restrictions:
\begin{gather*}
f_n = 2^f|_{F_{n}(X)},\\
f^{<\w} = 2^f|_{F(X)},
\end{gather*}
where $F_{n}(X) = \{ A \in 2^{X} \mid A \hbox{ has at most } n \text{ points} \}$ is the \emph{$n$-fold symmetric product of $X$} and $F(X) = \bigcup_{n = 1}^{\infty} F_{n}(X)$ is the collection of all finite subsets of $X$. Occasionally we shall write just $f$ for any of the above maps (including $2^f$ itself) as this does not lead to any confusion, and is useful to keep the notation simple, especially when we need to refer to the $k$\textsuperscript{th} iterate of the map $2^f$ which we simply denote by $f^k$.

The usual \emph{$n$-fold Cartesian product} will also be of interest as the $n$-fold symmetric product can be seen as a quotient of this space. Somewhat unconventionally we denote the product space $\underbrace{X \times X \times\cdots \times X}_{n-times}$ by $X^{(n)}$ and the induced map by $f^{(n)}$. This is not to be confused with $f^n$ which is simply the $n$\textsuperscript{th} iterate of $f$.

\section{Two Simple Results}\label{sec:interval}

To put ours result into perspective, we start with two related results. 

\begin{example}\label{ex:cylinder}
Let $C=\mathbb{S}^1\times[0,1]$ be a cylinder where $\mathbb{S}^1=[0,1]/_\sim$ denotes the unit circle obtained from the interval $[0,1]$ with its endpoints identified. Let $g\colon C\to C$ be an irrational rotation by $\alpha\in\R\setminus\Q$ about the central axis combined with an upward displacement that preserves the bases, e.g.\ $(\theta,z)\stackrel{g}{\mapsto} (\theta+\alpha \bmod{1},2z-z^2)$. Note that this is a homeomorphism of $C$. Let $X\subset C$ be the set consisting of the two bases $\mathbb{S}^1\times\{0,1\}$ and a full orbit $\{\dots,z_{-1},z_{0},z_{1},z_{2},\dots\}$ of a point $z_0$ on a generating line, say $z_0=(0,1/2)$, where we set $z_k=g^k(z_0)$ for all $k\in\Z\setminus\{0\}$ (see Figure \ref{fig:cyl}). Then clearly $f=g|_X$ has no periodic points and at the same time, for any $k\in \N$ the set $\mathbb{S}^1\times\{0,1\}\cup\{\dots,z_{-k},z_0,z_k,\dots\}$ is periodic under $2^f$ with period $k$.
\end{example}
\begin{figure}
\begin{tikzpicture}[x={(1,0)},z={(0,1)},y={(-.3,-.2)},scale=1.4]
\draw[thick] (1,0,1)
\foreach \z in {0,.1,...,10}
{ -- ({cos(\z*100)},{sin(\z*100)},{1})
};
\draw[thick] (1,0,-1)
\foreach \z in {0,.1,...,10}
{ -- ({cos(\z*100)},{sin(\z*100)},{-1})
};
\draw[dashed] ({cos(-100*6)}, {-sin(-100*6)}, {0.94*sin(-100)})
\foreach \z in {-100,-99.9,...,100}
{ -- ({cos(\z*6)}, {-sin(\z*6)}, {0.94*sin(\z)})
};
\foreach \k in {-99,-90,...,90}
{
\node at ({cos(\k*6)}, {-sin(\k*6)}, {0.94*sin(\k)}) {\tiny\textbullet};
};
\node at ({cos(0*6)}, {-sin(0*6)}, {0.94*sin(0)}) [right]{$z_0$};

\node at ({cos(9*6)}, {-sin(9*6)}, {0.94*sin(9)}) [right]{$z_1$};

\node at ({cos(18*6)}, {-sin(18*6)}, {0.94*sin(18)}) [above]{$z_2$};

\node at ({cos(27*6)}, {-sin(27*6)}, {0.94*sin(27)}) [above]{$z_{3}$};

\node at ({cos(-9*6)}, {-sin(-9*6)}, {0.94*sin(-9)}) [above]{$z_{-1}$};

\node at ({cos(-18*6)}, {-sin(-18*6)}, {0.94*sin(-18)}) [below]{$z_{-2}$};
\end{tikzpicture}
\caption{The set $X$ from Example \ref{ex:cylinder}.}
\label{fig:cyl}
\end{figure}
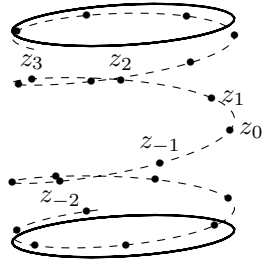

Block and Coven \cite{IntervalRecurrence} prove that if 
 $f$ is be a continuous map from a compact interval to itself and every point is chain recurrent, then either $f^2$ is the identity map or $f$ is turbulent. From this, one can easily deduce the following.

\begin{theorem}\label{tm:interval}
Let $f$ be a continuous map of a compact interval to itself. Then $\per{2^f}$ is either $\{1\}$ or $\{1,2\}$ or $\N$.
\end{theorem}

\begin{proof} If $f^2$ is the identity, then $\per{2^f}$ is either $\{1\}$ or $\{1,2\}$. If $f^2$ turbulent, then $f$ must have a periodic point with period which is not a power of $2$ (see p33 in \cite{BlockCoppel}).  But if $f$ has a point of period $2^n(2k+1)$, then by {\v{S}}arkovs{\cprime}ki{\u\i}'s Theorem, it has points of period $2^{n+1}m$ for any positive integer $m$, from which it follows that $2^f$ has points of period $m$, i.e. that $\per{2^f}=\N$. 

So suppose that there exists a point $x_0$ that is not chain recurrent.\footnote{Compare this with Proposition \ref{prop:nonrecurrent}. Also note that whether a point is \emph{(chain) recurrent} depends only on its forward orbit, as is common in the literature; but, for convenience, our notion of \emph{non-recurrent} point involves the full $\lambda$-limit set of the point. As a results, it is possible for a point to be neither recurrent nor non-recurrent.} Clearly, $x_0$ is not periodic. Also, we may assume that $f$ is onto, as otherwise one can take a restriction to the surjective core of the map $f$ and repeat the same reasoning. We can therefore take a full orbit $\bar{x}\in\inv{X}{f}$ of $x_0$ and consider the limit set
\begin{equation*}
\lambda(\bar{x})=\alpha(\bar{x})\cup\omega(\bar{x})=\bigcap_{m=0}^\infty\overline{\bigcup_{n=m}^\infty \{x_{-n},x_{n}\} }.
\end{equation*}
We claim that $x_0$ is non-recurrent, i.e.\ $x_0\notin \lambda(\bar{x})$.
Firstly note that $x_0\notin\w(x_0)$ as otherwise it would be recurrent and hence chain recurrent. Other possibility is that there exists an increasing subsequence $(p_k)_{k\in\N}$ such that $x_{-p_k}\to x_0$ as $k\to\infty$. But then $x_{-p_k+1}\to x_1$ and for any $\epsilon>0$ one can choose $k_0\in\N$ large enough so that $x_{-p_{k_0}+1}$ is $\epsilon$ close to $x_1$. Thus, $\{x_0, x_{-p_{k_0}+1}, x_{-p_{k_0}+2},\dots,x_0\}$ is an $\epsilon$-chain from $x_0$ to $x_0$ making $x_0$ a chain recurrent point. A contradiction. 

As $x_0$ is non-recurrent, by Proposition \ref{prop:nonrecurrent}, we immediately obtain all periods in $\per{2^f}$ proving our corollary.
\end{proof}

\section{Elementary periodic points}\label{sec:elementary}
We shall first describe the most basic type of periodic points that appear in the induced dynamics on $2^X$. It captures both the periods arising from cycles in $(X,f)$ via $x\mapsto\{x\}$ embedding as well as the periods as in Example \ref{ex:cylinder}.

Given a point $\bar{x}=(\dots,x_{-2},x_{-1},x_0,x_1,x_2,\dots)\in\inv{X}{f}$ we define the \emph{set of periods of $\bar{x}$} by
\begin{equation*}
\per{\bar{x}}=\{k\in\N \mid \overline{\{x_{mk} \mid m\in\Z\}}\in 2^X \text{ is periodic with period } k \}.
\end{equation*}
We use the same symbol as before but the meaning will be clear from the context. Note that this set does not depend on the choice of the starting point, i.e.\ $\per{\bar{x}}=\per{\s(\bar{x})}=\per{\s^{-1}(\bar{x})}$. It may however depend on the chosen backward orbit of $x_0$. For example, if $x_0$ is a fixed point which also has a history of infinitely many isolated points then $\per{(\dots, x_{-2}, x_{-1}, x_0,x_0,\dots)}=\N$ but $\per{(\dots,x_0,x_0,x_0,\dots)}=\{1\}$.

\begin{remark}
If $\{x_0,x_1,\dots,x_{p-1} \}$ is a $p$-cycle in $(X,f)$ then it is not hard to check that $\per{(\dots,x_{p-1},x_0,\dots,x_{p-1},x_0,\dots)}$ is the set of all divisors of $p$.
\end{remark}

The situation akin to that in Example \ref{ex:cylinder} occurs whenever there exists a \emph{non-recurrent point}, i.e.\ a point whose full orbit does not accumulate at the point itself. Formally $x\in X$ is \emph{non-recurrent} if there exists a full orbit $\bar{x}\in \inv{X}{f}$ of $x=x_0$ such that $x\notin \lambda(\bar{x})$.

\begin{proposition}\label{prop:nonrecurrent}
Let $f\colon X\to X$ be a continuous map and assume that there exists a non-recurrent point $x_0$ with a full orbit $\bar{x}=(\dots,x_{-1},x_0,x_1,\dots)\in \inv{X}{f}$. Then $\per{\bar{x}}=\N$ and therefore also $\per{2^f}=\N$.
\end{proposition}
\begin{proof}
Let $k\in\N$ be arbitrary and set $S=\overline{\{x_{mk} \mid m\in\Z\}}\in 2^X$. $S$ is clearly mapped to itself under $k$ iterates of $2^f$. It is therefore periodic with a fundamental period that divides $k$. To show that this is in fact $k$ it suffices to see that $x_0\notin (2^f)^n (S)$ for any $0<n<k$.

To that end take any such $n$ and note that $(2^f)^n (S)=\overline{\{x_{mk+n} \mid m\in\Z\}}$. Recall that $x_0\notin\lambda(\bar{x})$ and so if $x_0$ was in $(2^f)^n(S)$ it could only be equal to $x_{m_0 k+n}$ for some $m_0\in\Z$ and thus $x_0$ would have to be periodic. This would then imply that $x_0\in\w(\bar{x})\subseteq \lambda(\bar{x})$ which gives a contradiction.
\end{proof}
\begin{remark}
It could however happen that $x$ itself is recurrent but if it has a point $x_{-k}$ for some $k>0$ in the backward part of its orbit that is non-recurrent then the result still holds as $\per{\bar{x}}=\per{\s^{-k}(\bar{x})}=\N$.
\end{remark}

We have already used the obvious fact that $\per{\bar{x}}\subseteq \per{2^f}$ and one might suspect that $\bigcup_{\bar{x}\in\inv{X}{f}}\per{\bar{x}} = \per{2^f}$, but this is not the case. It suffices to take a system that consists of five points, two of which form a 2-cycle and the other three a 3-cycle. Then clearly $6\in \per{2^f}$ but no point in the inverse limit has 6 in its set of periods.

\medskip

$\per{\bar{x}}$, however, has a nice structure. Below we show that it is closed under taking divisors and least common multiples. This in particular implies that in case it is finite, $\per{\bar{x}}$ is simply the set of divisors of its largest element. But what if it is infinite, does it have to be $\N$? A negative answer to this provides the odometer, a classical example of a transitive system over the Cantor set.

\begin{example}
Let $X=\Sigma_2=\{0,1\}^{\N}$. The $2$-adic odometer $f\colon X \to X$ is defined recursively by
\begin{equation*}
f(\xi_0,\xi_1,\dots)=\begin{cases}
(1,\xi_1,\dots), \text{ if $\xi_0=0$,}\\
(0,f(\xi_1,\xi_2,\dots)), \text{ otherwise.}
\end{cases}
\end{equation*}
It is not hard to see that $\per{\bar{x}}=\{1,2,2^2,2^3,\dots\}$ is the set of all powers of $2$ for any $\bar{x}\in\inv{X}{f}$, and that $\per{2^f}$ is the same set.

We note in passing that this is also an example of a system where $2^f$-periodic points are dense in $2^X$. This is simply because any closed set in $\Sigma_2$ can be approximated by a finite union of clopen cylinders to an arbitrary precision and these cylinders are clearly periodic under $2^f$ and so is their finite union. We remind the reader that a cylinder in $\Sigma_2$ is a set of the form
\begin{equation*}
\{ (\xi_0,\xi_1,\dots)\in\Sigma_2 \mid \xi_i = a_i \text{ for } 0\le i \le n\},
\end{equation*}
for some $n\in\N$ and a choice of $a_i\in\{0,1\}$, and that all such sets form a clopen basis of the Cantor topology on $\Sigma_2$.
\end{example}

To prove the aforementioned structural result for elementary sets of periods, we shall need two lemmata.

\begin{lemma}\label{lm:containment}
If $S\in 2^X$ is a $k$-periodic then none of its $k-1$ iterates under $2^f$ can be a subset of $S$.
\end{lemma}
\begin{proof}
Otherwise, let us assume that $(2^f)^{j}(S)\subseteq S$ for some $0<j<k$. Using the usual notation one would write $f^{j}(S)\subseteq S$ and from there
\begin{equation*}
S=f^{kj}(S)\subseteq f^{(k-1)j}(S)\subseteq \dots \subseteq f^{j}(S)\subseteq S.
\end{equation*}
This makes $S$ periodic with a period strictly less than $k$, a contradiction.
\end{proof}

\begin{lemma}\label{lm:percharacterisation}
Let $\bar{x}\in\inv{X}{f}$. Then
\begin{equation}\label{eq:el}
\per{\bar{x}}=\{ k \in \N \mid \exists N_0\in\N \text{ s.t. } x_{-N_0k} \not\in \overline{\{x_l \mid l\in\Z \text{ and } k\ndiv l\}} \}.
\end{equation}
Also
\begin{equation}\label{eq:el2}
\per{\bar{x}}=\{ k \in \N \mid \exists N_0\in\N \text{ s.t. } \forall N\ge N_0 \quad x_{-Nk} \not\in \overline{\{x_l \mid l\in\Z \text{ and } k\ndiv l\}} \}
\end{equation}
and in particular
\begin{equation}\label{eq:union}
\per{\bar{x}}=\bigcup_{N=0}^\infty \{ k \in \N \mid  x_{-Nk} \not\in \overline{\{x_l \mid l\in\Z \text{ and } k\ndiv l\}} \}.
\end{equation}
\end{lemma}
\begin{proof}
\eqref{eq:el2} follows easily from \eqref{eq:el} if one recalls that $x_{-Nk}\not\in\overline{\{x_l \mid l\in\Z \text{ and } k\ndiv l\}}$ implies $x_{-(N+1)k}\not\in\overline{\{x_l \mid l\in\Z \text{ and } k\ndiv l\}}$. And from there it is clear that the sets under the union sign in \eqref{eq:union} form a monotonically increasing family of sets converging to the set in \eqref{eq:el}.

We now turn to proving that the fist claim holds. We first show that any $k$ satisfying the defining statement of the set in \eqref{eq:el} is necessarily in $\per{\bar{x}}$. Let such a $k\in \N$ be fixed and let $N_0$ be as in the definition of the set. It is clear that $x_{-N_0k}$ is a ``distinguishing feature'', an element contained in $\overline{\{x_{mk} \mid m\in\Z \}}$ but which cannot be in any of its $k-1$ forward iterates under $2^f$. It is thus truly a $k$-period point.

Conversely, take a $k\in\N$ which does not satisfy the statement in \eqref{eq:el}. Hence, there are infinitely many points in the $k$-step backward orbit of $x_0$ that are also in $\overline{\{x_l \mid l\in\Z \text{ and } k\ndiv l\}}$. In particular, infinitely many of them are in $\overline{\{x_{km+j} \mid m\in\Z\}}$ for some $0<j<k$. Similar reasoning as in the first paragraph of this proof allows us to conclude that, not just infinitely many, but all of the elements of the $k$-step orbit $\{x_{mk} \mid m \in \Z \}$ are contained in $\overline{\{x_{km+j} \mid m\in\Z\}}$ and hence their closure as well. But this implies that $\overline{\{x_{km} \mid m\in\Z\}}$ is not a $k$-periodic point by Lemma \ref{lm:containment}.
\end{proof}

\begin{proposition}\label{prop:eldiv}
Let $\bar{x}\in\inv{X}{f}$. Then $\per{\bar{x}}$ is non-empty (always contains at least $1$) and closed under taking divisors and least common multiples.
\end{proposition}
\begin{proof}
Clearly $1\in \per{\bar{x}}$. The rest is proved by invoking Lemma \ref{lm:percharacterisation}. Let $k\in \per{\bar{x}}$, $d|k$ and let $N_0$ be chosen as in \eqref{eq:el}, i.e.\ $x_{-N_0k} \not\in \overline{\{x_l \mid l\in\Z \text{ and } k\ndiv l\}}$ and hence $x_{-(N_0 k/d) d} \not\in \overline{\{x_l \mid l\in\Z \text{ and } k\ndiv l\}} \supseteq \overline{\{x_l \mid l\in\Z \text{ and } d\ndiv l\}}$. Therefore $d\in \per{\bar{x}}$.

As we have already seen that this set is closed under taking divisors it will suffice to show that for any two co-prime $m,n\in \per{\bar{x}}$ their product $mn$ is in there as well. Chose $N_0$ to be the greater of the two integers associated to $m$ and $n$ in the context of \eqref{eq:el}. Then $x_{-(N_0m)n}\not\in\overline{\{x_l \mid l\in\Z \text{ and } n\ndiv l\}}$ and $x_{-(N_0n)m}\not\in\overline{\{x_l \mid l\in\Z \text{ and } m\ndiv l\}}$. As $m$ and $n$ are co-prime we have that $nm\ndiv l$ if and only if $n\ndiv l$ or $m\ndiv l$, hence $x_{-N_0(mn)}\not\in\overline{\{x_l \mid l\in\Z \text{ and } n\ndiv l\}}\cup\overline{\{x_l \mid l\in\Z \text{ and } m\ndiv l\}}=\overline{\{x_l \mid l\in\Z \text{ and } mn\ndiv l\}}$. Therefore $mn \in \per{\bar{x}}$.
\end{proof}

\section{Admissible sets of periods for \texorpdfstring{$2^f$}{2\^{}f}}\label{sec:admis}

A few facts about $\per{2^f}$ are obvious. By considering singleton sets we see that $\per{2^f}\supseteq \per{f}$ and also $1\in \per{2^f}$ as $X\in 2^X$ is a fixed point. In fact $\per{f^{<\w}}=[\mathcal{D}(\per{f})]\footnote{See Remark \ref{rem:finiteperiods} below.}\subseteq \per{2^f}$ since $F(X)\subseteq 2^X$. It turns out that $\per{2^f}$ itself is closed under taking prime power divisors. The question if it must be closed under taking any divisor remains open. Surprisingly it needs not to be closed under taking least common multiples. The construction of such an example occupies Section \ref{sec:236}. Intuitively one might think that if $A_0,A_1,\dots, A_{n-1}$ is an $n$-cycle and $B_0,B_1,\dots, B_{m-1}$ is an $m$-cycle in $2^X$, and $d|n$ then $A_0\cup A_d \cup \dots A_{n-d}$ should be a $d$-periodic point and $A_0\cup B_0$ should be a $[m,n]$-periodic point. But this does not hold in general as it could happen that their fundamental period is smaller than expected. As a trivial example consider a map over $X=\{0,1,2,3\}$ given by $f(x)=x+1 \bmod{4}$. Then $A_0=\{0,1,2\}\in 2^X$ is $2^f$-periodic with period $4$ but $A_0\cup f^2(A_0)=X$ is a fixed rather than a period $2$ point in $2^X$.

\begin{theorem}\label{tm:primepowerdivisors}
Given a continuous map $f\colon X\to X$, the set of periods $\per{2^f}$ of the induced map on $2^X$ contains $\per{f} \cup \{1\}$ and is closed under taking prime power divisors.
\end{theorem}
\begin{proof}
We have already seen that $\per{f} \cup \{1\} \subseteq \per{2^f}$. Given an $n\in\per{2^f}$ and its prime factorisation $n=p_1^{\alpha_1} p_2^{\alpha_2}\cdots p_r^{\alpha_r}$ where $r\ge 1$, for the second part of the claim it will suffice to find a full orbit $\bar{x}\in\inv{X}{f}$ for which $\per{\bar{x}}$ contains $p_1^{\alpha_1}$ as the result will then follow from Proposition \ref{prop:eldiv}.

To that end we set $k=p_1^{\alpha_1}$ and $l=n/p_1=p_1^{\alpha_1-1} p_2^{\alpha_2}\cdots p_r^{\alpha_r}$, and let $A_0,\dots,A_{n-1}$ be a periodic orbit for $2^f$ of period $n$. By Lemma \ref{lm:containment} there exists $x_0\in A_0\setminus A_l$, and as $A_i$'s are mapped surjectively onto each other, it is possible to find a full orbit $\bar{x}=(\dots,x_{-1},x_0,x_1,x_2,\dots)$ of $x_0$ such that $x_{mn+i}\in A_i$ for all $m\in\Z$. We claim that $k\in\per{\bar{x}}$, i.e.\ that $\overline{\{x_{mk} \mid m\in\Z\}}$ is $k$-periodic under $2^f$.

Firstly note that $x_{-mn+i}\in A_i\setminus A_{i+k \bmod{n}}$ for $m\in\N$, as otherwise if $x_{-mn+i}\in A_{i+k \bmod{n}}$ then after mapping it forward by $f^{mn-i}$ we would have a contradiction with $x_{0}\in A_0\setminus A_{k}$. Similar reasoning allows us to conclude that, as we go backwards along the orbit, the pre-images $x_0,x_{-n},x_{-2n},x_{-3n},\dots$ belong to, possibly, more and more complements of different $A_i$'s and, as there are only finitely many of those, this number must stabilise. The backward iterate at which this happens is then taken to be $x_0$ and all the other indices are shifted accordingly. Note that this does not affect the claim we wish to prove as $\per{\bar{x}}=\per{\s^{t}(\bar{x})}$ for any $t\in\Z$. Also note that for this new $x_0$ we still have $x_0\in A_0\setminus A_l$. This modification will however allow us that from $x_0\in A_j$ for some $j$ we infer $x_{mn+i}\in A_{i+j \bmod{n}}$ for all $m\in\Z$ where before we could conclude this only for positive $m$'s. Essentially the same trick was used previously in the proof of Lemma \ref{lm:percharacterisation}.

After we altered the enumeration in $\bar{x}$ we are ready to conclude the proof by showing that (now modified) $x_0\notin\overline{\{x_t \mid t\in\Z \text{ and } k \ndiv t \}}$. Note that
\begin{equation*}
\overline{\{x_t \mid t\in\Z \text{ and } k \ndiv t \}}=\bigcup_{i=1}^{k-1}\overline{\{x_{mk+i} \mid m\in\Z\}}
\end{equation*}
and for the sake of getting a contradiction we assume that $x_0\in\overline{\{x_{mk+i} \mid m\in\Z\}}$ for some $0<i<k$. This means that either $x_0=x_{mk+i}$ for some $m\in\Z$ and therefore $x_0\in A_{\tilde{i}}$ where
\begin{equation*}
\tilde{i}\equiv mk+i\pmod{n},\quad \text{ hence }\quad \tilde{i}\equiv i\pmod{k},
\end{equation*}
or $x_0$ is a limit of such points ($x_{mk+i}$'s) out of which infinitely many must fall within the same congruence class with respect to $n$, say $\tilde{i}$, and thus also in the same $A_{\tilde{i}}$ for some $0\le\tilde{i}<n$, which must also satisfy $\tilde{i}\equiv i\pmod{k}$. As $A_{\tilde{i}}$ is closed, in both cases we get $x_0\in A_0\cap A_{\tilde{i}}$.

From here we can conclude that in the periodic case $x_0\in A_{\tilde{i}}\cap A_{2\tilde{i} \bmod{n}}$ as $x_0=x_{mk+i}=x_{2mk+2i}$. Also, if $x_0$ is a limit point of some subsequence in $\{x_{mn+\tilde{i}} \mid m\in\Z \}\subseteq A_{\tilde{i}}$, using the reasoning described above and in light of the fact that $x_0\in A_{\tilde{i}}$, this subsequence is also in $A_{2\tilde{i} \bmod{n}}$ and so must be $x_0$ as its limit. In both cases we henceforth conclude $x_0\in A_0\cap A_{\tilde{i}}\cap A_{2\tilde{i} \bmod{n}}$. Continuing in this fashion we see that $x_0 \in A_j$ for any $j$ that satisfies
\begin{equation*}
j \equiv s\tilde{i} \pmod{n} \text{ for some } s\in\N_0.
\end{equation*}
To obtain a contradiction with $x_0\notin A_l$ it now only remains to show that there exist $s\in\N_0$ for which $s\tilde{i}\equiv l\pmod{n}$. It is an elementary fact from number theory that there exists $\tilde{s}\in\N_0$ such that $\tilde{s}\tilde{i} \equiv (\tilde{i},n) \pmod{n}$, where $(a,b)$ stands for the greatest common divisor of integers $a$ and $b$. Note that $\tilde{i}\not\equiv 0 \pmod{k}$ and hence $p_1^{\alpha_1}=k\ndiv \tilde{i}$. Thus $(\tilde{i},n)=(\tilde{i},n/p_1)=(\tilde{i},l)\mid l$ and setting $s=\frac{l}{(\tilde{i},n)}\tilde{s}$ gives the desired conclusion.
\end{proof}

We now turn to showing that $\per{2^f}$ is not always closed under taking least common multiples. But for this we need first to revisit the notion of \emph{almost minimal systems}.

\section{Almost totally minimal Cantor system}\label{sec:atm}
Recall that any two non-compact, locally compact, totally disconnected, metrizable spaces with no isolated points are homeomorphic (see e.g.\ \cite{Danilenko-AlmostMinimal}*{Proposition 1.1}). These are essentially equal to the Cantor set without a point, which is in turn homeomorphic to a countable union of Cantor sets. We shall denote such a set by $\Cs^*$. In \cite{Danilenko-AlmostMinimal} Danilenko gives a direct limit construction of a class of minimal systems on $\Cs^*$. The system $(\Cs^*, T)$ is said to be \emph{minimal} if the only non-empty, closed, strongly $T$-invariant subset of $\Cs^*$ is $\Cs^*$ itself, or equivalently if any full orbit\footnote{If $T$ is non-invertible then there will points with more than one distinct full orbits.} of any point is dense in $\Cs^*$. In the compact case this is equivalent to asking that the $\w$-limit set of any point is whole of the state space, but here one has to be careful with the definition of minimality as some points can have their $\w$-limit sets empty due to non-compactness even if their full orbits are dense. In fact Danilenko proved that for any invertible minimal map on $\Cs^*$ the set of points with their forward orbit dense in $\Cs$ (those points $x$ for which $\w(x)=\Cs^*$) will be a dense, $G_\delta$ set with empty interior (see \cite{Danilenko-AlmostMinimal}*{Theorem 1.2}).

The class of maps constructed there conveniently extends to a class of continuous homeomorphisms of the Cantor set $\Cs$ which is obtained as one point compactification of $\Cs^*$ with infinity point being mapped onto itself. Such maps, with one fixed point and all the other points having dense full orbits are referred to as \emph{almost minimal systems} and in particular they are examples of \emph{essentially minimal systems} which are defined as those systems which have one unique minimal subsystem (see \cite{HPS}). To make our proofs in Section \ref{sec:236} work we need a map that is \emph{almost totally minimal} meaning that the system $(\Cs,T^m)$ is almost minimal for any iterate $T^m$ of the original map $T$ where $m\in\N$, in other words, that the system $(\Cs^*,T|_{\Cs^*})$ with the fixed point removed is totally minimal in the usual sense.

\begin{theorem}\label{tm:atm}
There exists an almost totally minimal homeomorphism $T \colon \Cs \to \Cs$ of the Cantor set $\Cs$. Such $T$ has exactly one fixed point $x^0\in\Cs$ and all the other points have their full orbit with respect to $T^m$ dense in $\Cs$, for any $m\in\N$ i.e.\
\begin{equation*}
(\forall x \in \Cs \setminus \{x^0\})\ (\forall m \in \N) \quad \clorbit{x}{T^m}=\Cs,
\end{equation*}
where $\clorbit{x}{T^m} = \clorbit{\bar{x}}{T^m} = \overline{\{ x_{mk} \mid k\in\Z \}}$ and $\bar{x}=(x_i)_{i\in\Z}\in\inv{\Cs}{T}$ is the full orbit of $x=x_0$.
\end{theorem}
\begin{remark}
We remark in passing that, as $\Cs$ is a Baire space without isolated points, $\clorbit{\bar{x}}{T}=\Cs$ is equivalent to $\lambda(\bar{x})=\Cs$ for any full orbit $\bar{x}\in\inv{\Cs}{T}$.
\end{remark}

To prove this theorem we revisit Gambaudo and Martens' combinatorial approach for constructing self-maps of the Cantor set (or indeed any $0$-dimensional compact metric space) via \emph{graph covers} (see \cite{GraphCovers}).

\subsection{Graph covers}
It is a well-known fact that any totally disconnected, compact and Hausdorff space (these are sometimes called \emph{Stone spaces}) can be obtained as an inverse limit of a (countable) system of discrete finite spaces. This property actually characterises Stone spaces via Stone duality (see \cite{Nagami}*{Proposition 8-5}) which is why these are also occasionally called \emph{profinite spaces}.

It turns out, using similar ideas, that it is possible to give a complete description not just for Stone spaces, but also for the self-maps on them by adding arrows to the discrete spaces forming the inverse system and, thus, creating an inverse limit of directed graphs. We briefly recall the main results of this theory following Shimomura's treatment in \cite{Shimomura}.

A \emph{graph}\footnote{All graphs we consider are directed.} is a pair $G=(V,E)$ where $V$ is a finite set of vertices and $E\subseteq V \times V$ is a set of directed edges. We additionally require that each vertex has at least one outgoing and one incoming edge. A \emph{graph homomorphism} between graphs $(V,E)$ and $(V',E')$ is a vertex map $\phi\colon V \to V'$ which respects the edges, i.e.\ for any pair $(v,w)\in E$ it must be $(\phi(v),\phi(w))\in E'$. A graph homomorphism is said to be \emph{$+$-directional} if $\phi(w_1)=\phi(w_2)$ whenever both $(v,w_1)\in E$ and $(v,w_2)\in E$. If additionally $(v_1,w)\in E$ and $(v_2,w)\in E$ implies $\phi(v_1)=\phi(v_2)$ it is said that $\phi$ is \emph{bidirectional}. A \emph{graph cover} is a +-directional homomorphism of graphs that is also \emph{edge-surjective} meaning that the map which $\phi$ naturally induces on the set of edges $\phi\colon E \to E'$ is surjective.

Given a sequence of graph covers $G_0 \get{\phi_0} G_1 \get{\phi_1} G_2 \get{\phi_2} \cdots$ one forms a Stone space as the inverse limit
\begin{equation*}
G_\infty =\varprojlim G_i= \left\{ (\sg{v}{i} )_{i\ge 0} \in \prod\limits_{i=0}^{\infty}V_i \mid \sg{v}{i}=\phi_i(\sg{v}{i+1}{}) \text{ for all } i\in\N_0 \right\}.
\end{equation*}
It is possible to define a self-map $\phi_\infty \colon G_\infty \to G_\infty$ by setting
\begin{equation*}
\phi_\infty \left((\sg{v}{i})_{i\ge 0}\right) = \left(\phi_i(\sg{w}{i+1})\right)_{i\ge 0}
\end{equation*}
where $\sg{w}{i}$ is any vertex for which the edge $(\sg{v}{i},\sg{w}{i})$ is in $E_i$. Intuitively, the map is given by the rule `follow the arrows' except that if there are more than one outgoing arrow from the chosen vertex at the any given level, we might need to peek at one level below whose finer resolution will help us decide which of the arrows to follow. +-directionality of the covers ensures that this process, or indeed $\phi_\infty$ is a well-defined continuous self-map of the Stone space $G_\infty$. If additionally each of the bounding maps in the sequence is bidirectional then the map $\phi_\infty$ is a homeomorphism, see \cite{Shimomura}*{Lemma 3.5}.

The full correspondence is given by

\begin{theorem}[\cite{Shimomura}*{Theorem 3.9}]
A topological (compact, metric, surjective) dynamical system is $0$-dimensional if and only if it is topologically conjugate to $G_{\infty}$ for some sequence of covers $G_0 \get{\phi_0} G_1 \get{\phi_1} G_2 \get{\phi_2} \cdots$.
\end{theorem}

We are now ready to construct the system described in Theorem \ref{tm:atm}. The inspiration for this came from the construction in \cite{ShimomuraScrambled}. Very recently similar techniques appeared in \cite{ShimomuraNew}.

\subsection*{Proof of Theorem \ref{tm:atm}}
Let $G_0=(\{\sg{v}{0}_0\},\{(\sg{v}{0}_0,\sg{v}{0}_0)\})$ be one vertex with a self-loop. Given $G_i$ we shall inductively define $G_{i+1}$. The graph $G_{i+1}$ will consist of $|V_{i+1}|=(|V_i|+C_i)\cdot (i+1)! +1$ vertices\footnote{Thus $|V_n|$ grows faster than superfactorials $n!(n-1)!\cdots 2! 1!$.} $V_{i+1}=\{\sg{v}{i+1}_0, \dots, \sg{v}{i+1}_{|V_{i+1}|-1}\}$ which make a cycle in that order, and additionally $E_{i+1}$ also contains a self-loop at $\sg{v}{i+1}_0$, see Figure \ref{fig:graphs}. The constants $C_i\in \N$ are inductively chosen so that $|V_i|+C_i$ and $(i+1)!$ are co-prime. It remains to specify the cover $\phi_i \colon V_{i+1} \to V_i$ by the formula
\begin{equation*}
\phi_i (\sg{v}{i+1}_k)=\sg{v}{i}_{I(i,k)}
\end{equation*}
where
\begin{equation*}
I(i,k)=
\begin{cases}
0, & \text{if } k\equiv 0, 1,2, \dots, \text{ or } C_i \pmod{|V_i|+C_i},\\
l, & \text{otherwise, where } l\equiv k-C_i \pmod{|V_i|+C_i}.
\end{cases}
\end{equation*}

It is perhaps easier to see what is going if we look at how $\phi_i$ maps the edges of $G_{i+1}$ onto the edges of $G_i$. Let us denote the self loops in $G_i$ and $G_{i+1}$ by $e_0$ and $e_0'$ respectively, and let $\{f_1,\dots, f_{|V_i|}\}$ and $\{f_1',\dots, f_{|V_{i+1}|}'\}$ be successive edges in their bigger cycles respectively. The formulae above capture the fact that the cycle $\{f_1',\dots, f_{|V_{i+1}|}'\}$ on the $(i+1)$\textsuperscript{st} level is wound $(i+1)!$ times over the full circuit (with the first edge being repeated $C_i$ times) in the graph $G_i$ in the following succession
\begin{equation*}
\underbrace{\underbrace{e_0 \dots e_0}_{C_i\; \text{times}} f_1 \dots f_{|V_i|} \;\; \underbrace{e_0 \dots e_0}_{C_i\; \text{times}} f_1 \dots f_{|V_i|} \quad \cdots \quad \underbrace{e_0 \dots e_0}_{C_i\; \text{times}} f_1 \dots f_{|V_i|}}_{(i+1)!\;\; \text{times}} \;\; e_0.
\end{equation*}
In particular, looking at Figure \ref{fig:graphs}, $\phi_1\colon G_2 \to G_1$ maps denoted edges as follows:
\begin{equation*}
\begin{array}{@{}*{11}{c}@{}}
 f_1',& f_2',& f_3',& f_4',& f_5',& f_6',& f_7',& f_8',& f_9',& f_{10}',& f_{11}'\\
\rotatebox[origin=c]{-90}{$\mapsto$}& \rotatebox[origin=c]{-90}{$\mapsto$}&\rotatebox[origin=c]{-90}{$\mapsto$}&\rotatebox[origin=c]{-90}{$\mapsto$}&\rotatebox[origin=c]{-90}{$\mapsto$}&\rotatebox[origin=c]{-90}{$\mapsto$}&\rotatebox[origin=c]{-90}{$\mapsto$}&\rotatebox[origin=c]{-90}{$\mapsto$}&\rotatebox[origin=c]{-90}{$\mapsto$}&\rotatebox[origin=c]{-90}{$\mapsto$}&\rotatebox[origin=c]{-90}{$\mapsto$}\\
e_0,& e_0,& f_1,& f_2,& f_3,& e_0,& e_0,& f_1,& f_2,& f_3,& e_0
\end{array}
\end{equation*}
Also note the additional twist over $e_0$ at the end which is needed to ensure that these are bidirectional covers.

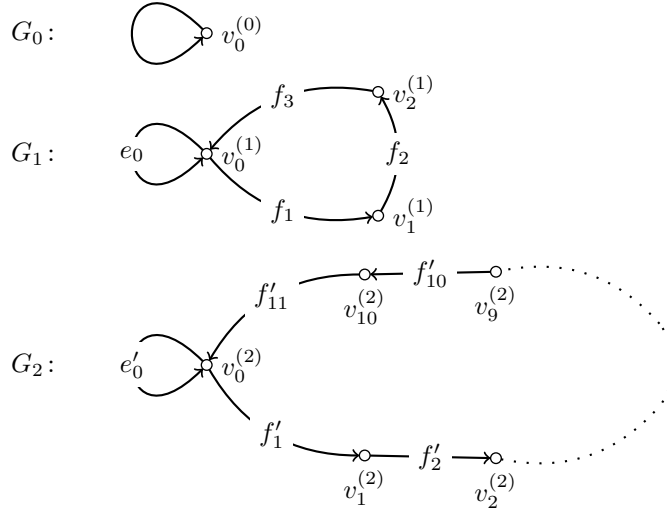
\begin{figure}
\GraphInit[vstyle=Welsh]
\tikzset{VertexStyle/.append style = {shape = circle,fill = white,inner sep = 0pt,outer sep = 0pt,minimum size = 4pt}}
\begin{tikzpicture}[x=0.8cm,y=0.8cm]
\begin{scope}
	\Vertex[L=$v^{(0)}_0$]{0}
	\node[left=50]{$G_0\colon$};

	\Loop[dist=50, dir= WE,style={->}](0)

\end{scope}
\begin{scope}[shift={(0,-2)}]
	\Vertex[L=$v^{(1)}_0$]{0}
	\node[left=50]{$G_1\colon$};
	\Vertex[a=-20,d=3, L=$v^{(1)}_1$, Lpos=0]{1}
	\Vertex[a=20,d=3, L=$v^{(1)}_{2}$, Lpos=0]{2}

	\Loop[label=$e_0$,labelstyle={fill=white},dist=50, dir= WE,style={->}](0)
	\Edge[label=$f_1$,style={->,bend right}](0)(1)
	\Edge[label=$f_2$,style={->,bend right}](1)(2)
	\Edge[label=$f_{3}$,style={->,bend right}](2)(0)
\end{scope}
\begin{scope}[shift={(0,-5.5)}]
	\Vertex[L=$v^{(2)}_0$]{0}
	\node[left=50]{$G_2\colon$};
	\Vertex[a=-30,d=3, L=$v^{(2)}_1$, Lpos=-90]{1}
	\Vertex[a=-18,d=5, L=$v^{(2)}_2$, Lpos=-90]{2}
	\Vertex[a=30,d=3, L=$v^{(2)}_{10}$, Lpos=-90]{10}
	\Vertex[a=18,d=5, L=$v^{(2)}_9$, Lpos=-90]{9}
	\Vertex[NoLabel, empty=true, a=-5, d=7.5]{R1}
	\Vertex[NoLabel, empty=true, a=5, d=7.5]{R2}

	\Loop[label=$e_0'$,labelstyle={fill=white},dist=50, dir= WE,style={->}](0)
	\Edge[label=$f_1'$,style={->,bend right}](0)(1)
	\Edge[label=$f_2'$,style={->}](1)(2)
	\Edge[label=$f_{10}'$,style={->}](9)(10)
	\Edge[label=$f_{11}'$,style={->,bend right}](10)(0)
	\Edges[style={loosely dotted,bend right}](2,R1,R2,9)
\end{scope}
\end{tikzpicture}
\caption{The first $3$ steps of the construction, with $C_0=1$ and $C_1=2$.}
\label{fig:graphs}
\end{figure}

Let $G_\infty$ be the inverse limit of the just constructed inverse system $G_0\get{\phi_0}G_1\get{\phi_1}G_2\get{\phi_2}\cdots$. As each vertex in $G_i$ is covered with at least $2$ vertices of $G_{i+1}$ one easily checks that $G_\infty$ has no isolated points and is therefore homeomorphic to the Cantor set. As we mentioned earlier, the sequence consists of bidirectional covers and $\phi_\infty \colon G_\infty \to G_\infty$ is thus a homeomorphism of the Cantor set.

It remains to be seen that $(G_\infty, \phi_\infty)$ is almost totally minimal. One fixed point of the system is clearly $(\sg{v}{i}_0)_{i\ge 0}$. Any other point $(\sg{v}{i}_{p_i})_{i\ge 0}\in G_\infty$ will have $p_i>0$ for all $i$ large enough. Let $x_\infty=(\sg{v}{i}_{p_i})_{i\ge 0}$ be one such a point. We wish to show that for any given $m\in\N$ the full orbit under $\phi_\infty^m$ of this point is dense in $G_\infty$. We fix one such $m$.

We shall denote by $[\sg{v}{i}_k]$ the set of all the points in $G_\infty$ with $i$\textsuperscript{th} coordinate equal to $\sg{v}{i}_k$. Recall that these sets $[\sg{v}{i}_k]$, where $i\in\N_0$ and $0\le k\le |V_i|-1$, form a clopen basis for the topology on $G_\infty$. They are neatly nested within each other and those on the level $i+1$ refine all those on levels less than $i$. In fact one has the relation $[v_{k}^{(i+1)}] \subseteq [\phi_i (\sg{v}{i+1}_k)]=[\sg{v}{i}_{I(i,k)}]$. It will thus suffice to see that for any level $L\in\N_0$, one can find $\phi_\infty^m$-iterates of $x_\infty$ that hit each of the sets $[\sg{v}{L}_k]$, for $0\le k \le |V_L|-1$.

Let $L$ be one such level and choose $M\ge\max\{L, m-1\}$ large enough so that $p_{M+1}>0$. We claim that $\phi_\infty^{(M+1)!}$-iterates of $x_\infty$ hit each of the sets $[\sg{v}{M}_k]$, for $0\le k \le |V_{M}|-1$, which will suffice as $m|(M+1)!$ and $M\ge L$.

Note that as $p_{M+1}>0$ it is possible to infer where a large portion of iterates of $x_\infty$ are mapped within level $M+1$ without actually having to look at what happens at the level below\footnote{By the level below we mean level $M+2$, i.e.\ a level with a finer structure.}. In particular we know that $\phi_\infty^{-p_{M+1} + k}(x_\infty) \in [\sg{v}{M+1}_k]$ for each $0\le k \le |V_{M+1}|-1$. This means that those same iterates hit each of the sets $[\sg{v}{M}_k]$, for $0\le k \le |V_{M}|-1$, at least $(M+1)!$ times, as each of the vertices in $G_{M+1}$ covers those of $G_{M}$ with at least multiplicity $(M+1)!$. This, along with the choice of $C_M$ (the number of successive repetitions of $\sg{v}{M}_0$ in the covering map), creates an apt offset implying that $\phi_\infty^{(M+1)!}$-iterates of $x_\infty$ end up in all of the sets $[\sg{v}{M}_k]$.

To put it precisely, let $p_{M+1} = q (M+1)! + r$ with $0 \le r<(M+1)!$ as in Euclidean division. Then if $k=s(M+1)!+r$ with $0\le s \le |V_{M}|+C_M-1$ we have $\phi_\infty^{(s-q) (M+1)!}(x_\infty) = \phi_\infty^{-q (M+1)! - r + k}(x_\infty) = \phi_\infty^{-p_{M+1} + k}(x_\infty) \in [\sg{v}{M+1}_k] \subseteq [\sg{v}{M}_{I(M,k)}]$. This is justified by noting that $k$ falls within the required range
\begin{equation*}
0\le r\le k\le (M+1)!(|V_M|+C_M)=|V_{M+1}|-1.
\end{equation*}

As $|V_{M}|+C_M$ and $(M+1)!$ are co-prime it is an elementary number theoretic fact that $s(M+1)!$, and hence also $k=s(M+1)!+r$, will run through all the residue classes modulo ${|V_{M}|+C_M}$ as $s$ runs through $\{0, 1, \dots, |V_{M}|+C_M-1\}$. Inspecting the definition of $I(M,k)=I(M,s(M+1)!+r)$, we see that this attains all the values in $\{0, 1, \dots, |V_{M}|-1\}$ when $s$ goes through $\{0, 1, \dots, |V_{M}|+C_M-1\}$. This completes our proof.\qed

\section{\texorpdfstring{$2, 3 \in \per{2^f} \notimplies 6 \in \per{2^f}$}{2, 3 in Per(2\^{}f) does not imply 6 in Per(2\^{}f)}}\label{sec:236}
Using Theorem \ref{tm:atm} we can now construct a map on the Cantor set $\Cs$ for which the induced map on the hyperspace $2^\Cs$ has period 2 and 3 points but no cycle of length 6.

Let $A  \sqcup B = \Cs$ be a partition\footnote{We use symbol $\sqcup$ to denote a disjoint union.} of the Cantor set in two clopen set such that the fixed point $x^0$ is in $A$. Let $\hat{X}=A\times\{0,1\} \sqcup B\times\{0\}$. We define a map $f\colon \hat{X} \to \hat{X}$ by
\begin{equation*}
f(x,i)=
\begin{cases}
(T(x),1), \text{ if } i=0 \text{ and } x\in T^{-1}(A), \\
(T(x),0), \text{ otherwise.}
\end{cases}
\end{equation*}
But this map is not surjective as for example $T(B)\cap A \times \{0\}$ is not contained in the range of $f$, and it is also $2$ to $1$ on $A \cap T^{-1}(B) \times \{0,1\}$ as $f(x,0)=f(x,1)=(T(x),0)$ for all $x\in A\cap T^{-1}(B)$. It is possible to restrict the dynamics to the surjective core of the map $f|_{\tilde{X}}\colon \tilde{X}\to \tilde{X}$ where $\tilde{X}=\bigcap_{k=0}^\infty f^k(\hat{X})$ is the set of all points in $\hat{X}$ that have preimages going infinitely back. This is still not the map we are after as $\tilde{X}$ is the largest closed strongly $f$-invariant set, and we would like our map to be defined on a minimal closed, strongly $f$-invariant set that contains $B\times\{0\}$. Such a set exists by Zorn's lemma, but it turns out that it is unique and can be constructed as follows.

Let
\begin{equation*}
N(x,B)=\min\{ k \in \N_0 \mid T^{-k}(x)\in B \},
\end{equation*}
be the time elapsed since $x$ last visited $B$. If the minimum above does not exist, we set $N(x,B)=\infty$. Clearly $N(x,B)=0$ for all $x\in B$, but it is also finite for all other $x\in\Cs$ which have their backward orbit dense in $\Cs$ which we know is a $G_\delta$ dense set in $\Cs$ with empty interior. In fact, $N(x,B)$ is finite on a much larger (dense and open) set $U=\{x\in\Cs \mid N(x,B)<\infty\}$. To see that $U$ is open it suffices to notice that $N(x,B)$ is a locally constant function. What we mean by this is that for any point $x\in U$ there exists a (cl)open neighbourhood containing $x$ on which the function $N(\;\cdot\;,B)$ is constant ($=N(x,B)$). This is easy to see as one just needs to ensure that the neighbourhood is small enough so that the $N(x,B)$ backward iterates of the points in it follow tightly those of $x$. We now set
\begin{equation*}
X=\overline{\{ (x,N(x,B)\bmod{2}) \mid x\in U \}}
\end{equation*}
where the closure is taken in $\hat{X}$. Note that $X$ is still homeomorphic to the Cantor set since it is defined as the closure of a subset of the Cantor set with no isolated points. It is also immediately clear that $X$ can be written as
\begin{equation*}
X=\overline{\bigcup_{k=0}^\infty f^k(B\times\{0\})}.
\end{equation*}

To prove that this is the unique minimal closed strongly $f$-invariant set containing $B\times\{0\}$ it only remains to be seen that $X\subseteq f(X)$. Actually, $B\times\{0\}\subseteq f(X)$ will suffice. For this we recall the natural factor map, the projection to the first coordinate, $\pi\colon \hat{X} \to \Cs$ given by $\pi(x,i)=x$. It is in fact a semi-conjugacy as one can verify that $\pi\circ f = T \circ \pi$. Using this we get $T(\pi(X))=\pi(f(X))\subseteq \pi(X)$ and so $\pi(X)$ is a $T$-invariant set clearly containing $B$. Remembering that the set of points with their forward orbit dense under $T$ is itself dense in $\Cs$ we conclude that $\pi(X)=\Cs$. But then given any $(b,0)\in B\times\{0\}$ set $b_{-1}=T^{-1}(b)$ and then since $\pi(X)=\Cs$, either $(b_{-1},0)$ or $(b_{-1},0)$ or both are in $X$. In any case one of these will map to $(b,0)$ under $f$ finishing our proof.

Another important thing to note is that the map $\pi$ restricted to $X$ is almost a homeomorphism as it is injective everywhere except possibly at some points of $\pi^{-1}(\Cs\setminus U)\cap X$ where it can be $2$ to $1$. To see that $\pi$ is injective on $\pi^{-1}(U)\cap X$ recall that for any $u\in U$ one can find a clopen neighbourhood $V_u$ where $N(\;\cdot\;,B)$ is constant and hence $V_u\times\{N(u,B)\bmod{2} \}$ is contained in $X$ but $V_u\times\{1- N(u,B)\bmod{2} \}$ has an empty intersection with $X$.

This observation allows us to prove
\begin{lemma}\label{lm:aboutX}
For any fixed natural number $m$, $X$ is the least closed $f^m$-invariant set that contains $B\times\{0\}$.
\end{lemma}
\begin{proof}
Let $m\in\N$ be fixed, and let $S\subseteq X$ be a closed, $f^m$-invariant subset of $X$ containing $B\times\{0\}$. We need to see that $S=X$. Reasoning similarly as before, we see that $\pi(S)$ is a closed $T^m$-invariant subset of $\Cs$ containing $B$. Thus, by Theorem \ref{tm:atm}, $\pi(S)=\Cs$. If we now remember that $\pi$ is injective on $\pi^{-1}(U)\cap X = \{ (x,N(x,B)\bmod{2}) \mid x\in U \}$, we get $\{ (x,N(x,B)\bmod{2}) \mid x\in U \}\subseteq S$, and from there immediately $S=X$.
\end{proof}
\begin{remark}
The statement of Lemma \ref{lm:aboutX} could equivalently be written as
\begin{equation*}
X=\overline{\bigcup_{k=0}^\infty f^{km}(B\times\{0\})},
\end{equation*}
for any $m\in\N$.
\end{remark}

\begin{lemma}\label{lm:transitivity}
For any $(x,i)\in X\setminus\pi^{-1}(x^0)$ and any $m\in\N$ we have
\begin{equation*}
\clorbit{(x,i)}{f^m} = X
\end{equation*}
for any chosen full orbit of $(x,i)$.
\end{lemma}
\begin{proof}
Using the fact that $\pi$ is a semi-conjugacy and Theorem \ref{tm:atm} we know that
$$\Cs=\clorbit{x}{T^m}=\clorbit{\pi(x,i)}{T^m}\subseteq\pi(\clorbit{(x,i)}{f^m}),$$
and hence $\pi(\clorbit{(x,i)}{f^m})=\Cs$.
But $\pi$ is injective on $B\times\{0\}$ and hence $B\times\{0\}\subset \clorbit{(x,i)}{f^m}$, and as $\clorbit{(x,i)}{f^m}$ is $f^m$-invariant and closed, by Lemma \ref{lm:aboutX} it must be that $\clorbit{(x,i)}{f^m}=X$.
\end{proof}

We are now ready to finish the construction of our example. To avoid confusion we abstain from denoting points in $X$ as ordered pairs, and we let the alternating 2-cycle be made of points $x^0, x^1\in X$. Let $Z=X\times\{0,1,2\}/_\sim$ be a quotient space obtained by gluing the points in $\{(x^0,i) \mid i=0,1,2 \}$ together and likewise the points in $\{(x^1,i) \mid i=0,1,2 \}$. Let us call those two points $z^0$ and $z^1$. We define the map $g\colon Z \to Z$ by
\begin{equation*}
g(x,i)=(f(x),i+1\bmod{3}).
\end{equation*}
Note that this rule respects the quotient relation and therefore accounts for a well-defined map on $Z$. Let us inspect closer the induced map $2^g$. It is clear that $\{z^0\}$ gives a 2-cycle in $2^Z$ as $z^0$ is 2-periodic for $g$. It is also clear that the natural embedding of $X$ in $Z$ as $X\times\{0\}$ produces a 3-cycle in $2^Z$. It remains to show that no 6-cycle in $2^Z$ exists. For if it did, it would have to contain at least one point $(x,i)\in Z$ other than $z^0$ or $z^1$. Then clearly $\clorbit{(x,i)}{g^6}$ would also have to be contained in this set for some choice of the full orbit of $(x,i)$. But $\clorbit{(x,i)}{g^6}=\clorbit{x}{f^6}\times\{i\}=X\times\{i\}$ and therefore our initial set must have been of the form $X\times F/_\sim$ for some $F\subseteq\{0,1,2\}$ and hence was not periodic with fundamental period $6$.

\subsection{General \texorpdfstring{$p$}{p} and \texorpdfstring{$q$}{q}}
Using the idea of the time elapsed since the last visit to $B$ it becomes obvious how to generalise this construction for $p$ and $q$ other than 2 and 3. We set
\begin{equation*}
X=\overline{\{(x,N(x,B)\bmod{p}) \mid x\in U \}},
\end{equation*}
and define a map $f\colon X \to X$ by
\begin{equation*}
f(x,i)=
\begin{cases}
(T(x),i+1\bmod{p}), \text{ if } x\in T^{-1}(A), \\
(T(x),0), \text{ otherwise.}
\end{cases}
\end{equation*}
This is a continuous map on the Cantor set $X$ with a $p$-cycle $(x^0,0)$,$(x^0,1)$, \dots, $(x^0,p-1)$. As before, using the projection map $\pi\colon X\to \Cs$, one checks that $f$ is onto, and that $\pi(X)=\Cs$. Analogously one sees that $\pi$ is injective on $\pi^{-1}(U)\cap X=\{(x,N(x,B)\bmod{p}) \mid x\in U \}$, and hence any closed subset of $X$ that projects onto the whole of $\Cs$ must be $X$ itself. This enables one to prove that Lemmata \ref{lm:aboutX} and \ref{lm:transitivity} still hold even with this, more general, definition of $X$ and $f$. In fact, their statements along with the proofs (up to changing $2$ to $p$) still hold word for word.

One can now set $Z=X \times \{0,1,\dots, q-1\}/_{\sim}$ where $\sim$ identifies $q$ distinct $p$-cycles into one $p$-periodic orbit $\{z^0,\dots, z^{p-1}\}\subset Z$. The map $g\colon Z \to Z$ given by
\begin{equation*}
g(x,i)=(f(x),i+1\bmod{q})
\end{equation*}
is once again well-defined and clearly satisfies $p,q\in \per{2^g}$. We claim that any other $m\in \per{2^g}$ must be a divisor of either $p$ or $q$.

Firstly, if we assume that the periodic point $S\in 2^Z$ is completely contained in $\{z^0,\dots, z^{p-1}\}$ then by our previous results we know that $m|p$. If otherwise there exists a point $(x,i)\in S$ other than any of $z^0,\dots z^{p-1}$ then $\clorbit{(x,i)}{g^{mq}}$ is contained in $S$ for some choice of the full orbit of $(x,i)$. By Lemma \ref{lm:transitivity} we get $\clorbit{(x,i)}{g^{mq}}=\clorbit{(x)}{f^{mq}}\times \{i\}=X\times\{i\}$ and so $S=X\times F/_\sim$ for some $F\subseteq\{0,1,\dots,q-1\}$. From here we immediately get that $m|q$. Thus we have proved
\begin{theorem}\label{tm:pq}
Let $p,q \in \N$. There exists a continuous onto map of the Cantor set $g\colon Z \to Z$ for which the periods appearing in the induced map are exactly divisors of either $p$ or $q$, i.e.\
\begin{equation*}
\per{2^g}=\{m \mid m|p \text{ or } m|q\}.
\end{equation*}
\end{theorem}

\section{Periodicity in symmetric products}\label{sec:per}

In this section we shall be concerned only with periodic points in $2^X$ made up entirely of periodic points for $(X,f)$. We shall see that in this setting the anomalies such as those in Theorem \ref{tm:pq} are not possible.

Firstly note that one can reduce the problem to studying periodic points of $f^{<\w}$. Namely, let $S\in 2^X$ be a periodic point with fundamental period $k=p_1^{\alpha_1}\cdots p_r^{\alpha_r}$ where each point of $S$ is periodic under $f$. Then as in the proof of Theorem \ref{tm:primepowerdivisors} we can find points $x_1$, \dots, $x_r$ in $S$ for which $p_i^{\alpha_i} \in \per{\bar{x}_i}$ where each $\bar{x}_i$ is simply a full periodic orbit of $x_i$. It could happen that some of these $x_i$'s are represented by the same point or by points that are members of the same full orbit, which in this case means the same cycle. For a moment let us assume that this is not the case and that all of these $r$ full orbits are mutually disjoint when considered as subsets of $X$. Then for each $1\le i \le r$ let $S_i\in 2^X$ be the canonical $k_i=p_i^{\alpha_i}$-periodic set given by $S_i=\{x_i, f^{k_i}(x_i), f^{2 k_i}(x_i), \dots\}$. Note that $S_i$'s are finite and mutually disjoint, and thus $\bigcup_{i=1}^r S_i$ is a $k$-periodic point for $2^f$ but also for $f^{<\w}$.

Suppose now that $x_1$ and $x_2$ belong to the same orbit. Then using Proposition \ref{prop:eldiv} we can actually conclude that $l_1=k_1k_2\in\per{\bar{x}_1}=\per{\bar{x}_2}$. Grouping the elements of the same orbits in this way and discarding all but one representative for each of the orbits we obtain a factorisation of $k=l_1\cdots l_s$ in $s\le r$ co-prime factors where each $l_i$ is an element of some $\per{\bar{x}_{t(i)}}$ with $x_{t(1)}$, \dots, $x_{t(s)}$ all belonging to distinct orbits. After setting $S_i=\{x_{t(i)}, f^{l_i} (x_{t(i)}), f^{2l_i}(x_{t(i)}), \dots \}$ the union $\bigcup_{i=1}^s S_i$ will again be a $k$-periodic point for both $2^f$ and $f^{<\w}$.

\medskip

Below we give explicit formulae for $\per{f^{(n)}}$, $\per{f_{n}}$, and $\per{f^{<\w}}$ in terms of $\per{f}$ where by $\per{f}$ we denote the \emph{set of periods} of the periodic points of $f$.

\begin{proposition}
\label{prop:periods}
We have the following identities
\begin{gather}
\per{f^{(n)}} =\left\{ [m_1,\dots,m_n] \mid m_i\in \per{f} \textnormal{ for all }
1\le i\le n \right\},\label{eq:p1}\\
\per{f_{n}} = \bigcup_{l=1}^n\left\{ [d_1,\dots,d_l] \mid d_i|m_i\in \per{f}
\textnormal{ for all } 1\le i\le l, \phantom{\frac{m_l}{d_l}}\right.\label{eq:p2}\\
\left.\hspace{15em}\textnormal{ and } \frac{m_1}{d_1}+\dots+\frac{m_l}{d_l}\le n\right\},\nonumber\\
\per{f^{< \w}} =\bigcup_{l=1}^\infty\left\{ [d_1,\dots,d_l] \mid d_i|m_i\in \per{f}
\textnormal{ for all } 1\le i\le l \right\},\label{eq:p3}
\end{gather}
where $[k_1,\dots,k_m]$ stands for the least common multiple of $k_1,\dots,k_m$.
\end{proposition}

\begin{proof}
Statement \eqref{eq:p1} is easy. Given points $q_1,\dots, q_n$ with fundamental periods
$m_1, \dots, m_n$ it is clear that the point $(q_1,\dots,q_n)\in X^{(n)}$
has fundamental period $[m_1,\dots,m_n]$. And conversely any periodic point
of $f^{(n)}$ must arise in this fashion.

We shall now prove \eqref{eq:p2}. Let $Q=\{q_1,\dots, q_n\}$ be a periodic point in
$F_n(X)$. Clearly each $q_i$ must be a periodic point of $f$. Then $Q$
can be naturally partitioned into sets $Q_1, \dots Q_l$ where each $Q_i$
contains points belonging to the same cycle under $f$, and no two
points from distinct $Q_i's$ belong to the same cycle. All the points in
$Q_i$ have the same period $m_i$ under $f$. Let us denote the
periods of $Q,Q_1,\dots,Q_l\in F_n(X)$ by $d,d_1,\dots,d_l$
respectively. Clearly $d=[d_1,\dots,d_l]$ and since $f_n^{m_i}(Q_i)=Q_i$
we also have $d_i | m_i$. By the construction we have $l\le n$ as every $Q_i$
must contain at least one point from $Q$.

Lastly, we claim that $\frac{m_i}{d_i}\le |Q_i|$ where $|\, . \,|$ denotes
the cardinality of a set. This is because $f_n^{d_i}$ acts on $Q_i$ as
a permutation over set of cardinality $|Q_i|$ of order at most $\frac{m_i}{d_i}$.
But a permutation can not have order larger than the cardinality of the set it
acts upon and, thus, we get the claim. From this it follows
$\sum_{i=1}^{l}\frac{m_i}{d_i}\le |\bigcup_{i=1}^{l} Q_i|=|Q|=n$ which
finishes the proof of the inclusion $\subseteq$ in \eqref{eq:p2}. The other
inclusion follows by noting that any period of the form $[d_1,\dots,d_l]$
where $1\le l\le n$, $d_i | m_i$ for some $m_i\in \per{f}$, and
$\sum_{i=1}^{l}\frac{m_i}{d_i}\le n$ can be realised by taking points
$r_1,\dots,r_l$ of periods $m_1, \dots, m_l$ respectively and then
setting $Q=\bigcup_{i=1}^l Q_i$ where for each $1\le i \le l$  the set
$Q_i$ is formed of $\frac{m_i}{d_i}$ equispaced iterates of $r_i$.
More precisely $Q_i=\{r_i,f^{d_i}(r_i),f^{2d_i}(r_i),\dots,f^{(\frac{m_i}{d_i}-1)d_i}(r_i)\}$.

Finally, note that $\per{f^{<\w}}=\bigcup_{n\in\N} \per{f_n}$, therefore \eqref{eq:p3}
follows immediately from \eqref{eq:p2}.
\end{proof}

\begin{remark}
An important observation that implicitly drives the proof of the second statement above was that no periods are lost if one only considers points in $F_n(X)$ formed of equispaced\footnote{Formally, we say that $Q=\{q_1,\dots,q_m\}$ is an equispaced point in $F_n(X)$ if $\min\{k \ge 1 \mid f^k(q_i)\in Q\}$ does not depend on $1\le i\le m$.} iterates of points in $X$. To take a simple example, consider a point $x$ which is periodic with fundamental period $6$ in some system $(X,f)$. Even though $\{x,f(x)\}$, $\{x,f(x), f^2(x)\}$, and many others are period $6$ points for $f_3$, there is a simple, equispaced point of period $6$ for $f_3$, namely $\{x\}$.
\end{remark}
\begin{remark}\label{rem:finiteperiods}
Note that statement \eqref{eq:p3} could equally be written as
\begin{equation*}
\per{f^{<\w}}=\left[\mathcal{D}\left(\per{f}\right)\right]
\end{equation*}
where $\mathcal{D}(S)=\{d\in\N \mid d|n \text{ for some } n\in S \}$ is the set of divisors of $S$, and $[S]=\{[k_1,\dots,k_l] \mid l\in\N, k_1, \dots k_l \in S \}$ is the set of least common multiples of a set $S\subseteq \N$. In particular, $\per{f^{<\w}}$ is the smallest set containing $\per{f}$ that is closed under taking least common multiples and divisors.
\end{remark}

Let us show how Proposition \ref{prop:periods} can be used to compute the periods of induced maps on hyperspaces.
\begin{example}
Let $f\colon[0,1]\to[0,1]$ be a $2^r$ interval map ($r>1$), i.e.\ a map for which $\per{f}=\{1,2,2^2,\dots 2^{r-1}, 2^r \}$. One possible construction of such a map is described in \cite{BlockCoppel}*{Example I.13}. For any positive integer $n$, all the sets $\per{f_n}$, $\per{f^{(n)}}$, $\per{f^{<\w}}$ clearly contain $\per{f}$. Note that in our special case, if $d_i|m_i$ and $m_i\in \per{f}$ then $d_i\in \per{f}$, and hence also $[d_1,\dots,d_l]=\max\{d_1,\dots,d_l\}\in \per{f}$. Therefore $\per{f_n}=\per{f^{(n)}}=\per{f^{<\w}}=\per{f}$. Note that using Theorem \ref{tm:interval} and the fact $\per{f}\subseteq\per{2^f}$ we can also conclude $\per{2^f}=\N$.
\end{example}

\section{Concluding remarks}\label{sec:concluding}

Nearly all the results contained in here were driven by the problem of finding a complete characterisation of the set $\per{2^f}$ in terms of $\per{f}$. The first step would be to find constraints on $\per{2^f}$ and this already seems involved. We have shown that this set must be closed under taking prime divisors, but the question remains if it is closed under taking divisors in general. Whatever turns out to be the truth, one will still be faced with the laborious task of constructing maps attaining all the admissible periodicity sets in-between, if one seeks for such a full characterisation. Some of these will be easy, but we have already gone through much trouble in order to construct examples having two co-prime periods in $\per{2^f}$ but not their product. We conjecture that, more generally, a map with $\per{2^f}=\{m \mid m|p \text{ for } p\in P \}$ exists for any finite $P\subset \N$.

One plausible strategy for obtaining such results easily would be to find a way to embed any Cantor set dynamics $(X,f)$ within another (``bigger'') Cantor set dynamics $(\hat{X},\hat{f})$ where $X$ would be a closed, nowhere dense set in $\hat{X}$. Additionally one would require that $(\hat{X}\setminus X,\hat{f})$ is totally minimal. More precisely, we formulate this as
\begin{question}
Let $(X,f)$ be a dynamical system on the Cantor set $X$. Does there always exist a system on the Cantor set $(\hat{X},\hat{f})$ and an embedding (continuous injection) $i\colon X \to \hat{X}$ which:
\begin{enumerate}[label=(\alph*)]
\item intertwines the dynamics of $f$ and $\hat{f}$, i.e.\ $\hat{f}\circ i = i \circ f$,
\item embeds $X$ as an $\hat{f}$-invariant set whose complement in $\hat{X}$ is also invariant and non-empty,
\item and such that for any $m\in\N$ and any full orbit $\bar{x}\in\inv{\hat{X}}{\hat{f}}$ of a point $x=x_0$ in $\hat{X}\setminus i(X)$ we have $\clorbit{\bar{x}}{\hat{f}^m}=\hat{X}$?
\end{enumerate}
Note that these imply that the embedding $i(X)$ is nowhere dense in $\hat{X}$.
\end{question}

Using this one could, starting with our $p$, $q$ example as $(X,f)$, construct a system with $\per{2^f}$ consisting solely of  divisors of $p$, $q$ and $r$ by taking the set $\hat{X}\times \{0,1,\dots,r-1\}/_\sim$, where $\sim$ identifies $r$ copies of $X$ into just one copy, and the map $(x,i)\mapsto (\hat{f}(x),i+1\bmod{r})$. This procedure can be then repeated with the system just obtained as a starting point. Continuing in this way would enable one to construct systems with an increasing hierarchy of subsystems nested within each other. It is not hard to check that $\per{2^f}$ of such a map would be exactly of the form above.

\bibliographystyle{plain}
\bibliography{refs}

\appendix

\section{Bratteli-Vershik representation}

We have mentioned before that it would be possible to give rewrite rules for translating a description of a systems given by graph covers to those using Bratteli-Vershik diagrams or Kakutani-Rokhlin towers. In this section we give an example of how this is done by providing another, equivalent construction of the system $(\Cs, T)$ from Theorem \ref{tm:atm} using the former.

\medskip

There are many references explaining the theory behind Bratteli-Vershik diagrams, e.g.\ \cites{HPS, MarkovOdometers}, but to us most relevant was a recent preprint \cite{Category} by Amini, Elliott, and Golestani where they give a category theory treatment of the matter highlighting the connection between these diagrams and essentially minimal systems.

\medskip

Consider the infinite graph in Figure \ref{fig:bv}, where all the edges are oriented downwards even if this is not depicted there. Additionally, for each node other than the root $L_0$ which has no incoming edges, an ordering is given on the set of incoming edges into that node. This order is depicted in the figure as the order in which the edges connect to that node going from left to right.

We have yet to explain how this infinite graph is constructed. Recall that in the proof of Theorem \ref{tm:atm} we obtained a sequence of numbers $(C_i)_{i\in\N_0}$ which was of some significance in the construction of the map $T$. Here, any node $L_{i+1}$ will have only one incoming edge coming from the node $L_{i}$ for any $i\in\N_0$. The incoming edges into the node $R_{i+1}$ will be: $C_{i+1}$ edges coming form the node $L_{i}$, then one edge from the node $R_i$, then another $C_{i+1}$ edges form $L_{i}$, and another one from $R_i$, and this sequence of $C_{i+1}+1$ edges is to be repeated in total $(i+1)!$ times, after which the last one in comes the edge from $L_i$. Note that we could simply say that there are $(i+1)!\cdot(C_{i+1})+1$ edges coming from $L_{i}$ and $(i+1)!$ from $R_{i}$ into $R_{i+1}$, but by listing them as above, we implicitly specified the order on those edges incoming into the node $R_{i+1}$ for any $i\in\N$.

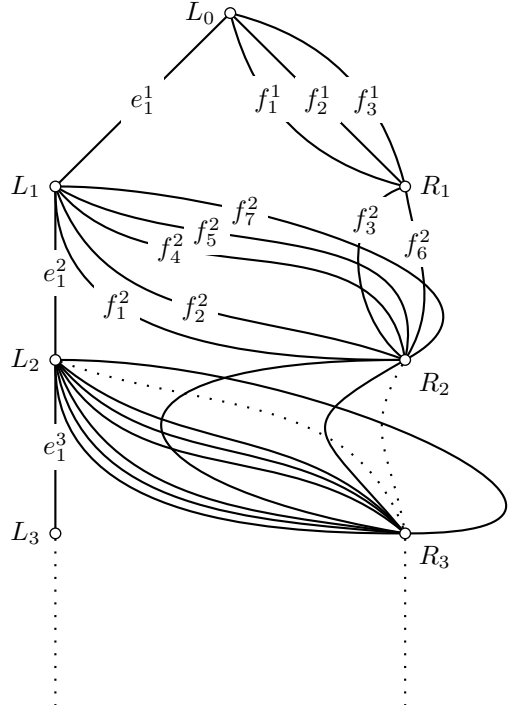
\begin{figure}
\begin{center}
\GraphInit[vstyle=Welsh]
\tikzset{VertexStyle/.append style = {shape = circle,fill = white,inner sep = 0pt,outer sep = 0pt,minimum size = 4pt}}
\usetikzlibrary{positioning}
\begin{tikzpicture}[trim left=(1), trim right = (2)]
\begin{scope}
\SetGraphUnit{2.3}
	\Vertex[L=$L_0$, Lpos=180]{0}
	\SOWE[L=$L_1$, Lpos=180](0){1}
	\SOEA[L=$R_1$](0){2}
	\SO[L=$L_2$, Lpos=180](1){3}
	\SO[L=$R_2$,Lpos=-40](2){4}
	\SO[L=$L_3$, Lpos=180](3){5}
	\SO[L=$R_3$,Lpos=-40](4){6}
	\SO[NoLabel,empty=true](5){7}
	\SO[NoLabel,empty=true](6){8}
		
	\Edge[label=$e^1_1$](0)(1)
	\Edge[label=$e^2_1$](1)(3)
	\Edge[label=$e^3_1$](3)(5)
	\Edge[style={loosely dotted}](6)(8)
	\Edge[style={loosely dotted}](5)(7)
	
	\Edge[label=$f^1_1$,style = {bend right=30,pos=.35}](0)(2)
	\Edge[label=$f^1_2$,style = {bend right=0,pos=.5}](0)(2)
	\Edge[label=$f^1_3$,style = {bend right=-30,pos=.65}](0)(2)
	
	\Edge[label=$f^2_1$,style = {out=-90, in=180,pos=.35}](1)(4)
	\Edge[label=$f^2_2$,style = {out=-70, in=160}](1)(4)
	\Edge[label=$f^2_3$,style = {out=-160, in=140,pos=.26}](2)(4)
	\Edge[label=$f^2_4$,style = {out=-50, in=100,pos=.3}](1)(4)
	\Edge[label=$f^2_5$,style = {out=-30, in=80,pos=.32}](1)(4)
	\Edge[label=$f^2_6$,style = {out=-80, in=60,pos=.3}](2)(4)
	\Edge[label=$f^2_7$,style = {out=0, in=30, distance=60,pos=.32}](1)(4)

	\Edge[style = {out=-90, in=180}](3)(6)
	\Edge[style = {out=-80, in=175}](3)(6)
	\Edge[style = {out=-70, in=172}](3)(6)
	\Edge[style = {out=-180, in=170, distance=120}](4)(6)
	
	\Edge[style = {out=-60, in=140}](3)(6)
	\Edge[style = {out=-50, in=135}](3)(6)
	\Edge[style = {out=-40, in=132}](3)(6)
	\Edge[style = {out=-150, in=130, distance = 50}](4)(6)
	
	\Edge[style = {out=-20, in=110, loosely dotted}](3)(6)
	\Edge[style = {out=-120, in=100, distance = 30, loosely dotted}](4)(6)
	
	\Edge[style = {out=0, in=0, distance=100}](3)(6)
\end{scope}
\end{tikzpicture}
\end{center}
\caption{Bratteli-Vershik representation of the system from Theorem \ref{tm:atm}.}
\label{fig:bv}
\end{figure}

The space we shall be considering is the space of all infinite paths starting from $L_0$ and then following the sequence of edges down through the vertices $L_i$/$R_i$. Encoding these paths as sequences of edges allows one to see this space as the Cantor set of symbolic sequences. It remains to specify the map on this space of infinite paths.

Recall that for each node other than the root there exists the minimal and the maximal ingoing edge into that node. A path is said to be minimal (resp.\ maximal) if it consists entirely of minimal (resp.\ maximal) edges. It is not hard to check that in our graph the path $e^1_1e^2_1e^3_1\dots$ that is always staying on the left side is both minimal and maximal and no other minimal nor maximal path exits. This path will be the fixed point of our map. For any other path $p^1p^2p^3\dots$ there must exist a $k\in\N_0$ such that the edge $p^k$ is not maximal. Choose this $k$ to be the smallest possible and then we let our function map this path to the path $q^1q^2\dots q^kp^{k+1}p^{k+2}p^{k+3}\dots$ where $q^k$ is the successor of the edge $p^k$ and $q^i$'s for $i<k$ are chosen in such a way that $q^1q^2\dots q^{k-1}$ forms the (unique) minimal path connecting $L_0$ and the source vertex of the edge $p^{k+1}$. Equivalently, $q^1\dots q^k$ is chosen to be the successor of the path $p^1\dots p^k$  with respect to the natural induced order on the sequence of paths (of length $k$) terminating at the source of $p^{k+1}$.

We leave to the reader to check that this does give the same (or formally, conjugated) system as the method given in Section \ref{sec:atm}. For this it is helpful to see how this representation was derived from Figure \ref{fig:graphs}.

Each vertex above stands for a cycle in the graph covers representation. To be precise, the self-loop of the graph $G_0$ is represented by the root $L_0$, the self-loop of any $G_i$ is represented by the vertex $L_i$, and the bigger cycle $\sg{v}{i}_0\sg{v}{i}_1\dots\sg{v}{i}_{|V_i|}$ in $G_i$ corresponds to the right side vertex $R_i$. Finally the edges and the order in which they connect vertices $L_i$ and $R_i$ to the vertices $L_{i+1}$ and $R_{i+1}$ are to be inferred from the way the cycles in the graph $G_{i+1}$ wrap around through the cycles of $G_i$. In particular, the loop $e_0'$ in $G_2$ just goes once through the loop $e_0$ in $G_1$ and hence only one edge from $L_1$ to $L_2$. On the other hand, the cycle $f_1'f_2'\dots f_{11}'$ in $G_2$ wraps firstly two times around $e_0$, then once around the bigger cycle in $G_1$, then repeats this, and then finally winds the last time over $e_0$. This means that the edges incoming into $R_2$ are: two edges from $L_1$, followed by one from $R_1$, then again two from $L_1$ and one from $R_1$, and lastly one from $L_1$.

\medskip

At the very end we mention that the corresponding representation using Ka\-ku\-ta\-ni-Rokhlin towers is obtained by associating the cycles in the graphs $G_i$ with the eponymous towers and the vertices within those cycles correspond to the levels of the towers.
\end{document}